\documentclass[12pt,reqno, a4paper]{amsart}

\usepackage[legalpaper, margin=1.4in]{geometry}
\usepackage[colorlinks = true,
linkcolor = blue,
urlcolor  = blue,
citecolor = blue,
anchorcolor = blue]{hyperref}  
\usepackage[english]{babel}
\usepackage{amsmath, amssymb}
\usepackage[latin1]{inputenc}
\usepackage{import}
\usepackage{epsfig}
\usepackage{color}
\usepackage{graphics}
\usepackage{mathrsfs} 
\usepackage{relsize}
\usepackage{exscale}
\usepackage{graphicx}

\newtheorem{theorem}{Theorem}[section]
\newtheorem{proposition}[theorem]{Proposition}
\newtheorem{lemma}[theorem]{Lemma}
\newtheorem{corollary}[theorem]{Corollary}
\newtheorem{remark}[theorem]{Remark}

\theoremstyle{definition}

\newtheorem{example}[theorem]{Example}

\numberwithin{equation}{section}

\def\R{\mathbb{R}}                              
\def\C{\mathbb{C}}                              

\def\CalC{\mathcal{C}}

\def\CalO{\mathcal{O}}

\def\B{{\Bbb B}}

\def\T{{\Bbb T}}

\begin{document}

	\title[Truncated tube domains with multi-sheeted envelope]
	{Truncated tube domains with multi-sheeted envelope}

	\author{Suprokash Hazra}
	\curraddr{Department of Mathematics\\ Mid Sweden University\\
		SE-851 70 Sundsvall\\ Sweden} \email{suprokash.hazra@miun.se, Egmont.Porten@miun.se}
	\author{Egmont Porten}


	\begin{abstract}	
		The present article is concerned with a group of problems raised by J.~Noguchi and M.~Jarnicki/P.~Pflug, 
		namely whether the envelopes of holomorphy of truncated tube domains are always schlicht, i.e.~subdomains of $\C^n$,
		and how to characterise schlichtness if this is not the case.
		By way of a counter-example homeomorphic to the $4$-ball, we answer the first question in the negative.
		Moreover, it is possible that the envelopes has arbitrarily many sheets.
		The article is concluded by sufficient conditions for schlichtness in complex dimension two.		
	\end{abstract}
	
	\maketitle
	
	\section{Introduction}
	A domain $D\subset\C^n$ is said to be a \emph{tube domain} 
	if there are domains $X$, $Y\subset\R^n$
	such that $D=X+iY$. 
	In case that one of the domains $X$, $Y$ equals $\R^n$, $D$ is a \emph{Bochner tube domain}.
	Otherwise $D$ is often called \emph{truncated} tube domain. 
	A classical theorem of Bochner \cite{B} shows the envelope of holomorphy of $X+i\R^n$ 
	to be the tube $\mbox{conv}(X)+i\R^n$ over the convex hull $\mbox{conv}(X)$, 
	meaning in particular that the envelope is schlicht.
	
	Truncated tube domains are much less understood. 
	While the geometry of pseudoconvex tube domains is still simple 
	(see \cite{P} and \cite{N} for an alternative approach),
	the investigation of the envelope becomes subtle in the nonpseudoconvex case.
	Recently, questions about multi-sheetedness of these envelopes were raised
	by  J.~Noguchi \cite{N} and by M.~Jarnicki and P.~Pflug \cite{JP2}.
	More precisely, it is asked in \cite[Rem.4]{JP2} whether these envelopes are always schlicht,
	and in \cite[Problem 4.4]{N} whether simply connected or contractible tube domain have schlicht envelope.
	In the present article, we will present counter-examples answering both of the questions in the negative,
	namely tube domains homeomorphic to the $4$-ball 
	whose envelopes have arbitrarily many sheets.
	
	Let us review some terminology and results on envelopes, see \cite{JP1} for exhaustive information. 
	It is classical that every domain $D\subset\C^n$ admits a
	Riemann domain $\pi:\textsf{E}(D)\rightarrow\C^n$ 
	(unique up to isomorphisms of Riemann domains over $\C^n$)
	with the following properties:
	\begin{itemize}
		\item[(i)\phantom{ii}] 
		There is an embedding $\alpha:D\hookrightarrow\textsf{E}(D)$
		with $\pi\circ\alpha=\mbox{id}_D$.
		\item[(ii)\phantom{i}] For every $f\in\CalO(D)$, the lift $f\circ(\pi|_{\alpha(D)})
		$ extends holomorphically to $\textsf{E}(D)$,
		\item[(iii)] $\textsf{E}(D)$ is Stein.
	\end{itemize}
	This domain is called the \emph{envelope of holomorphy of $D$}, see \cite{JP1} for exhaustive information.
	We say that $\textsf{E}(D)$ is \emph{$k$-sheeted over $z\in\C^n$} 
	if $\# \pi^{-1}(z)=k$. 
	For $U^{\rm open}\subset\C^n$, a sheet lying \emph{uniformly over $U$}
	is an open set $V\subset\textsf{E}(D)$ such that $\pi|_V$ is a homeomorphism onto $U$.
	The latter notion is mainly interested for an infinite number of sheets.
	If $\textsf{E}(D)$ is $k$-sheeted over $z\in\C^n$ with $k<\infty$,
	then there are $k$ sheets lying uniformly over any sufficiently small
	neighbourhood of $z$.
	
	The following is our main result.
	
	\begin{theorem}\label{main}
		For every integer $k\geq 2$ there is a domain $X_k\Subset\R^2$
		with $\CalC^\infty$-smooth boundary and a radius $r_k>0$ such that
		the envelope of holomorphy of $D_k=X_k+iB_{r_k}$ has at least $k$ sheets
		over some point in $\C^2$. 
		Furthermore, there are a domain $X_\infty\Subset\R^2$ and $r_\infty>0$ 
		such that $\textsf{E}(X_\infty+iB_{r_\infty})$ has infinitely many sheets lying uniformly 
		over a neighbourhood of a circle embedded in $\C^2$. 
		The domains $X_k$ and $X_\infty$ can be constructed diffeomorphic to the open $4$-ball.
	\end{theorem}
	
	Throughout, $B_r(y_0)$ denotes the round ball $\{y\in\R^n:|y-y_0|<r\}$ in $\R^n$,
	in constrast to balls in $\C^n$, denoted by $\B_r(z_0)$.
	If $y_0=0$, we abreviate $B_r=B_r(0)$. 
	Note however that the exact shape of the domains in $y$-direction is not very important.
	Moreover, it is easy to derive corresponding examples of tube domains in $\C^n$, $n\geq 3$.
	More interesting is the difference concerning boundary regularity for $X_k$, $k<\infty$, and $X_\infty$.
	The authors do not know whether the envelope is always finitely sheeted 
	if both $X$ and $Y$ are bounded and have smooth boundary.
	This may be viewed as a minor variant of the 
	long-standing open problem raised by B.~Stens\o nes (Problem 3.3.1, \cite{prob})
	whether envelopes of bounded domains with smooth boundary are finitely sheeted.
	
	On the other side, one may still look for sufficient conditions ensuring schlichtness.
	In recent work, M.~Jarnicki and P.~Pflug \cite{JP2}  explicitly determined the envelopes of the model domains
	\begin{equation}\label{model}
	D_{n,r_1,r_2,r_3}=\{x\subset\R^n:r_1<|x|<r_2\}+iB_{r_3}\subset\C^n, 
	\end{equation}
	where $n\geq 2, 0\leq r_1<r_2,\ 0<r_3$, to be schlicht and equal to
	\begin{equation}\label{ED}
	\textsf{E}(D_{n,r_1,r_2,r_3})=\{|x|<r_2,|y|<r_3,|y|^2<|x|^2-(r_1^2-r_3^2)\}.
	\end{equation}
	Motivated by their question about generalizations, we present a general schlichtness theorem for $n=2$.
	
	\begin{theorem}\label{schlicht2d}
	Let $X\subset\R^2$ be a convex domain with finitely many  holes with strictly convex $\CalC^2$-boundary,
	and let $Y\Subset\R^2$ be a convex domain.
	Then the envelope of holomorphy of $D=X+iY$ is schlicht.
	\end{theorem}
	
	Strict convexity in the above statement is meant in the sense of nonzero curvature,
	see Section 4 for details.
	Note that the domains $D_{2,r_1,r_2,r_3}$ satisfy the assumptions with $X=\{r_1<|x|<r_2\}$.
	In the general situation, one cannot expect to obtain a description of $\textsf{E}(D)$
	as explicit as in (\ref{ED}).  
	Actually, Lemma \ref{polhull} shows that this is essentially as hard as describing polynomial hulls.
	However, Theorem \ref{schlicht2d} permits to derive some qualitative consequences.
	Our proof essentially relies on phenomena special to complex dimension two,
	and the corresponding question in higher dimension remains open.
	 	
%
	
	The article is organised as follows. In Section 2, we construct the counter-examples mentioned in the main result.
	The idea will be to produce a helix with many turns in the envelope of holomorphy
	and lying over a circle in $\C^2$. 
	The question whether this helix can arranged to be maximal in the sense of universal coverings, 
	is discussed and answered in Section 3. Moreover, the relation of our construction to limit cycles of laminations will become apparent.
	Section 4 finally contains the proof of Theorem \ref{schlicht2d} and the discussion of some structural consequences.\\

	{\bf Acknowledgements:} The authors would like to thank Professor Peter Pflug for drawing their attention to his recent article
	and for a lot of valuable remarks helping to ameliorate the quality of the article. 
	They would also like to thank Professor Junjiro Noguchi for his interest in their work.
	
	\section{Proof of the main result}\label{sec2}
	The present section is dedicated to the {\bf proof of Theorem \ref{main}}, organised in four major steps.\\
	
	{\bf Step 1: Geometric preparations.}
	For a complex line $L=\{az_1+bz_2=0\}\subset\C^2$, $(a,b)\not=(0,0)$,
	the logarithm $f_L=\log(az_1+bz_2)$ can be viewed as
	a univalent holomorphic function on the universal covering 
	$\pi_{\rm univ}:Z\rightarrow\C^2\backslash L$
	of $\C^2\backslash L$.
	Of course, $Z$ equals essentially $Z_{\log}\times\C$ where $Z_{\log}$ is the Riemann surface
	of the logarithm in one complex variable.
	More precisely, if another complex line $L'$ meets $L$ transversally at the origin, 
	it is clear that $\pi_{\rm univ}^{-1}(L')$ is the universal covering of $L'\backslash\{0\}$
	(and is equivalent to $Z_{\log}$).
	
	Recall that the set $\mathcal T$ of all real planes $\Pi$ intersecting $L$ transversally at $0$
	is a topological open $4$-ball. 
	To see this, we pick one such plane $\Pi_0$ and select real linear coordinates $s_1,s_2,t_1,t_2$
	such that $L=\{s_1=s_2=0\}$, $\Pi_0=\{t_1=t_2=0\}$. Then the planes $\Pi\in{\mathcal T}$ are represented as graphs 
	\[
	\left(\begin{array}{c} t_1 \\ t_2\end{array}\right)=A(\Pi)\left(\begin{array}{c} s_1 \\ s_2\end{array}\right),
	\]
	establishing a 1-1 correspondence between $\Pi\in{\mathcal T}$ and the real $2\times 2$ matrices $A(\Pi)$.

	An easy homotopy argument shows 
	that we can replace $L'$ above by any $\Pi\in{\mathcal T}$.
	In particular, $Z$ is still connected with infinitely many sheets over $\Pi\backslash\{0\}$.
	If $\gamma$ is a loop surrounding $0$ once
	within $\Pi$, following $\gamma$ one turn takes us to a new sheet 
	and changes thus the value of $\log(az_1+bz_2)$ by $\pm 2\pi i$.
	
	We summarise what we need in the sequel.
	
	\begin{lemma}
	Let $L\subset\C^2$ be a complex line and $\Pi\subset\C^2$ a real $2$-plane 
	which intersect transversally at the origin. 
	Let $\iota:\T=\{\zeta\in\C:|\zeta|=1\}\hookrightarrow\Pi\backslash\{0\}$ be a continuous embedding
	that surrounds $0$ in $\Pi$ precisely once, 
	$D\subset\C^2$ a simply connected domain containing $\iota(\T\backslash\{1\})$ but not $\iota(\{1\})$,
	and $F_L$ a continuous branch of $f_L$ on $D$. 
	Then the two one-sided limits 
	$\lim_{\theta\downarrow 0}F_L(e^{i\theta})$ and $\lim_{\theta\uparrow 2\pi}F_L(e^{i\theta})$
	exist and differ by $\pm2\pi i$.
	\end{lemma}
	In more invariant terms, a general loop in $\C^2\backslash L$
lifts as a loop to $Z$ if and only if it is not linked with $L$, but we will not need this.
		
	
	\begin{remark}\rm
	{\bf a)}
	The constructions in this and the next section rely on the initial choice of a pair $(L,\Pi)$
	of a complex line $L$ and a real $2$-dimensional subspace $\Pi$ transverse to $L$.
	These pairs are the points of a $6$-dimensional manifold $\mathcal{P}$.
	We claim that $GL_2(\C)$ acts transitively on $\mathcal{P}$, 
	meaning that all initial choices are equivalent. 
	Given $(L,\Pi)$, $j=1,2$, it is easy to reduce to the case that 
	$\Pi_1=\Pi_2=\R^2_x$ and 
	$L_j={\C\left(\!\!\begin{array}{c}1\\a_j+ib_j\end{array}\!\!\right)}$, $b_j\not=0$.
	We conclude by observing that (i) the matrices in $GL_2(\R)$ fix $\R^2_x$
	and (ii) $L_1$ can always be mapped to $L_2$ by a real invertible matrices of the form 
	$\left(\!\!\begin{array}{cc}1 & 0\\c&d\end{array}\!\!\right)$. 
	
	{\bf b)}
	The larger manifold  $\mathcal{Q}$ of all pairs $(\Pi_1,\Pi_2)$ where the $\Pi_j$ 
	are transverse real $2$-dimensional subspaces of $\C^2$ is $8$-dimensional.
	By dimensional reasons 
	($GL_2(\C)$ has real dimension $8$, but real multiples act the same way),
	the action of $GL_2(\C)$ is not transitive. 
	Alternatively, one may also observe that pairs of the form $(\Pi,i\Pi)$
	form a $4$-dimensional orbit or use
	\cite{W}, where it is shown that $\Pi_1\cup\Pi_2$ 
	is sometimes locally polynomially convex and sometimes not.
	\end{remark}
	
	{\bf Step 2: A first option for $D_2$.}
	In the above setting we can take $\Pi=\R^2_x$ 
	and fix a complex transverse line $L$ accordingly (for example $L=\{z_1=iz_2\}$).
	Set $A=\{(x_1,x_2)\in\R^2:1<|x|<2\}$, and define
	\begin{equation}\label{R}
	R=\max\{r>0:L\cap(A+iB_{r})=\emptyset\}.
	\end{equation}
	To show that $R$ is well defined, we easily verify that
	$\{r>0:L\cap(A+iB_{r})\not=\emptyset\}$ is a nonempty 
	interval that is either of the form $(0,r_0]$ or $(0,\infty)$.
	The last option is impossible, since any real plane through the origin
	different from the $i\R^2$ must meet $A+iB_{r}$, 
	provided $r$ is large enough.
	
	Working in $\R^2$, we observe that the circle $\{|x-(3/2,-2)|=2\}$
	intersects the closure of $A$ in two connected arcs.
	Denoting the component passing through $p=(3/2,0)$ by $\gamma_0$,
	we set $X=A\backslash\gamma_0$.
	
	Pick $r\in(0,R]$. We claim that $\textsf{E}(X\times iB_r)$ is at least $2$-sheeted over the points 
	below the hypersurface $H_0=\gamma_0+iB_r$ and close to $(3/2,0)$.
	To prove this, we observe that  $H_0$ is strictly pseudoconvex.
	By the Levi extension theorem, holomorphic functions locally defined on the upper side
	$\{|x-(3/2,-2)|>2\}$ extend across $H_0$ to a uniform neighbourhood of $(3/2,0)$. 
	We can choose a branch $f$ of $f_L$ on the simply connected domain $X+iB_r$.
	Our initial remarks on $f_L$ imply that the extension of $f$ 
	through $H_0$ differs by $\pm 2\pi i$ from $f$ on the side $|x-(3/2,-2)|<2$,
	and the claim follows. What we have shown so far is already enough to answer 
	Noguchi's question in the negative.

To finish the construction of $D_2$, we have to smoothen the boundary $X$.
This can for example be done by cutting out from $X$ the component of 
\[
\{2-\eta\leq|x-(3/2,-2)|\leq 2\}\cap X, \mbox{ where } 0<\eta\ll 1,
\]
containing $\gamma_0$ and rounding off the four corners.
Note that the choice of $\eta$ sensitively depends on $r$,
as will become clear from the discussion of the Levi extension theorem
before the proof of Lemma \ref{Ginf}.\\

{\bf Step 3: Construction of $D_\infty$.}
To produce envelopes with infinitely many sheets,
we will replace the previous annulus-type domain by a spiral-shaped subdomain of $\{|x|>1\}$ 
with the circle $S=\{|x|=1\}$ in its limit set. The mapping
\[
\varphi:[1/2,1]\times\R_{\geq 0}\rightarrow\R^2, (s,\theta)\mapsto(1+se^{-\theta})(\cos\theta,\sin\theta)
\]
is an embedding of the semi-infinite strip $[1/2,1]\times\R_{\geq 0}$ into $\{|x|>1\}$ (use that $e^{-2\pi}<1/2$). 
The domain $X_\infty=\varphi((1/2,1)\times\R_{>0})$ is squeezed between the straight segment
with endpoints $(3/2,0)$ and $(2,0)$ and the two spirals 
\[
\Gamma_{1/2}=\varphi(\{1/2\}\times\R_{\geq 0})\ \mbox{ and }\ \Gamma_{1}=\varphi(\{1\}\times\R_{\geq 0}).
\]

\vspace*{-0.55cm}
\begin{figure}[!htb]\label{fig1}
	\centering
		\def\svgwidth{\linewidth}
		\fontsize{12pt}{1em}
		\scalebox{.8}{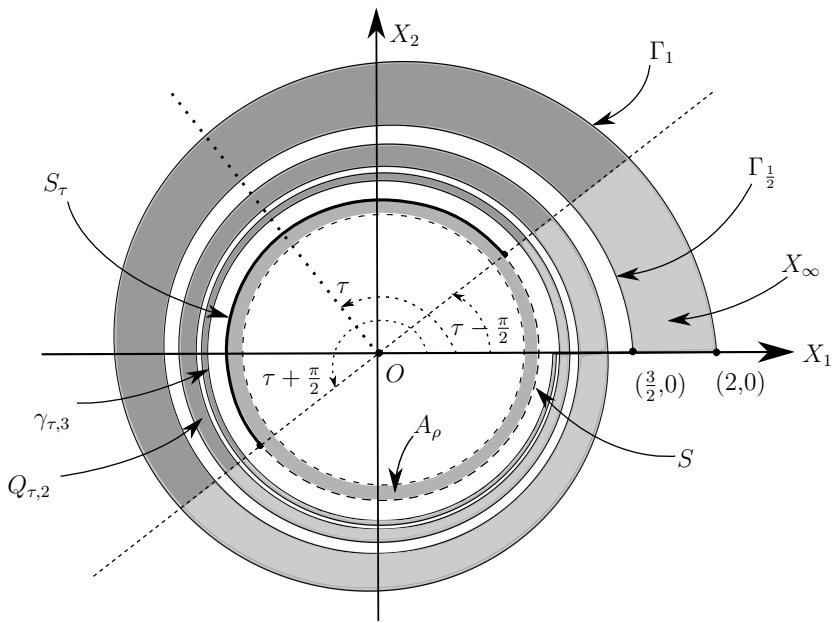}
		\vspace*{-0.4cm}
		\caption{Intersection of $D_{\infty}$ with $\R^2$}
	\end{figure}
Note that $S$ also belongs to $\overline{X}_\infty$, meaning in particular 
that the set theoretic boundary $\partial X_\infty$
is not smooth along $S$, see Figure 2.1. 
\begin{lemma}\label{Ginf}
Set $D_r=X_\infty+iB_{r}$ with $r\in(0,R)$. Then there is $\rho>0$ such that $\textsf{E}(D_r)$
contains a Riemann subdomain which has infinitely many sheets over every point of 
the thickened  annulus $A_{\rho}=\{1-\rho<|x|<1\}+iB_{\rho}$.
\end{lemma}

Before entering the proof of Lemma \ref{Ginf}, some comments are in order.\\

\noindent
{\bf (1)} Let us identify the universal covering of $A_{\rho}$ with 
\[
\{1-\rho<s<1\}\times\R_\theta\times B_{\rho}
\]
via the map $(s,\theta,y)\mapsto(s\cos\theta,s\sin\theta,y)$.
Then the proof below will show 
that the Riemann subdomain in the lemma can actually be chosen as the lifting of a domain of the form 
\begin{equation}\label{lift}
\{1-\rho<s<1\}\times\{\theta>C\}\times B_{\rho},
\end{equation}
with $C>0$ sufficiently large.\\

\noindent
{\bf (2)} 
Recall the Levi extension theorem:
\it 
Let $H$ be a connected closed real hypersurface of a domain $G\subset\C^2$ 
such that $G\backslash H$ has two connected components $G^\pm$. 
In addition, assume that the Levi form of $H$ is nondegenerate in every point 
and that $G^+$ is the pseudoconcave side of $H$. 
For $z_0\in H$, there is an $\eta>0$ with the following properties:
For every open set $V$ containing $H$, all functions holomorphic 
on $V^+=V\cap G^+$ extend holomorphically to 
$V^+\cup \big(B_\eta(z_0)\cap\overline{G^-}\,\big)$.\\

\rm
We need two additional stability properties of Levi extension:
{\bf (i)} Since it mainly depends on Levi curvature, it is stable under $\CalC^2$-small deformations of $H$.
More concretely, the $\eta$ above can be chosen uniformly for $H$ 
and all sufficiently $\CalC^2$-close deformations fixing $z_0$.
{\bf (ii)} The constant $\eta$ does not depend on the size of $G^+$. 
If $G$ is replaced by a smaller domain $G_1$ with analogous properties, 
we still obtain holomorphic extension to $G_1^+\cup\big(B_\eta(z_0)\cap\overline{G^-}\,\big)$.
These properties can be shown by the method of analytic discs, see \cite{MP1}.\\

{\bf Proof of Lemma \ref{Ginf}:}
The spiral $\Gamma_{1/2}$ accumulates smoothly to $S$.
More precisely, for any interval $(\theta_1,\theta_2)\subset\R_{>0}$ which is \emph{short} 
in the sense that $\theta_2-\theta_1<2\pi$, the parametrisations
\[
\varphi_j(\theta)=\varphi(1/2,\theta+2\pi j),\ \theta_1<\theta<\theta_2,\ j=0,1,\ldots,
\]
tend to $(\cos\theta,\sin\theta)$ on $(\theta_1,\theta_2)$ with all their derivatives for $j\rightarrow\infty$.

Fix now $r<R$ like indicated in the assumptions. For $\tau\in[\pi/2,5\pi/2)$, 
we look at the short interval
$I_{\tau}=(\tau-\pi/2,\tau+\pi/2)$, 
the corresponding halfcircle $S_{\tau}\subset S$
centered at $z_{\tau}=(\cos\tau,\sin\tau)$,
and the strictly pseudoconvex hypersurface $H_{\tau}=S_{\tau}+iB_r$.
For $j=0,1,2,\ldots$, we define the translated intervals 
\[
I_{\tau,j}=(\tau+\pi(2j-1/2),\tau+\pi(2j+1/2))
\]
and the curved rectangles
\[
Q_{\tau,j}=\varphi((1/2,1)\times I_{\tau,j}))+iB_r.
\]
The corresponding tubes
$Q_{\tau,j}+iB_r$
have a pseudoconcave boundary part
\[
H_{\tau,j}=\varphi(\{1/2\}\times I_{\tau,j}))+iB_r,
\]
which converges in $\CalC^2$-sense toward $H_\tau$ for $j\rightarrow\infty$.


Set $z_{\tau,j}=\varphi(1/2,\tau+2\pi j)$.
For $\eta>0$ not to large ($\eta<1/4$ will do), 
the disc $B_{3\eta/2}(z_{\tau,j})$ is disconnected by the arc
$\gamma_{\tau,j}=\varphi(\{1/2\}\times I_{\tau,j})$
into two connected components $B^\pm_{3\eta/2}(z_{\tau,j})$.
We choose signs such that $B^-_{3\eta/2}(z_{\tau,j})$ 
lies on the pseudoconvex side of $H_{\tau,j}$. We now define the enlarged domain
\[
Q_{\tau,j}(\eta)=Q_{\tau,j}\cup\Big(\big(\gamma_{\tau,j}\cap B_{3\eta/2}(z_{\tau,j})\big)+iB_\eta \Big)
\cup \Big(B^-_{3\eta/2}(z_{\tau,j})+iB_\eta\Big).
\]
The stability properties of Levi extension discussed above yield an $\eta=\eta(\tau)>0$
such that for $j$ sufficiently large, say $j\geq J=J(\tau)$, the following holds:
Any function holomorphic on $Q_{\tau,j}$ extends to $Q_{\tau,j}(\eta)$.
Enlarging $J$ if necessary, we may in addition assume that the domain
$Q_{\tau,j}(\eta)$, $j\geq J$, contains 
\[
\B^-_\eta(\cos\tau,\sin\tau)=\B_\eta(\cos\tau,\sin\tau)\cap (B_1+i\R^2),
\]
where $\B_\eta(z)$ denotes the round ball in $\C^2$.
By compactness, we can choose $J$ and $\eta$ 
uniformly for $\tau\in[\pi/2,5\pi/2]$. 

Relabelling the domains $Q_{\tau,j}$ as $Q_\sigma$ if $\sigma=\tau+2\pi j$,
we define $Q_\sigma$ for any $\sigma\geq\pi/2$.
Then we obtain for every $f\in\CalO(D_\infty)$ 
a $\sigma$-dependent family of  germs 
$f_\sigma\in\CalO_{\C^2,(\cos\sigma,\sin\sigma)}$\footnote{
$\CalO_{\C^2,(z)}$ denotes the stalk of holomorphic germs at the point $z\in\C^2$.},
$\sigma\geq 2J\pi+\pi/2$, by 
\begin{itemize}
\item[\bf (i)\phantom{ii}] restricting $f$ to $f|_{Q_\sigma}$, 
\item[\bf (ii)\phantom{i}] extending $f|_{Q_\sigma}$ to $Q_\sigma(\eta)$, and
\item[\bf (iii)] restricting to the germ of the extension at $(\cos\sigma,\sin\sigma)\in Q_\sigma(\eta)$.
\end{itemize}
This yields a continuously parametrised curve 
$\sigma\mapsto p(\sigma)\in\textsf{E}(D_\infty)$, $\sigma\geq 2\pi J+\pi/2$,  over $S$. 

Next we specialise to a branch $\tilde{f}_L$ of $f_L=\log(az_1+bz_2)$ on $D_\infty$,
which exists since $D_\infty$ is simply connected.
The above construction also gives a continuous curve
$\sigma\mapsto q(\sigma)\in Z$ over $S$, $\sigma\geq 2\pi J+\pi/2$, 
by mapping the germ of the extension at $(\cos\sigma,\sin\sigma)$ 
to the corresponding point in the Riemann domain of $\tilde{f}_L$.
With every turn, $q(\sigma)$ changes sheet, and therefore the values
of two consecutive extensions at $(\cos\sigma,\sin\sigma)$ differ by $\pm 2\pi i$.
It follows that $p(\sigma)$ also changes sheet with every turn,
and
Lemma \ref{Ginf} is proved. \qed\\

{\bf Step 4: Construction of $D_k$.}
With $J$ and $\eta$ as chosen in the proof of Lemma \ref{Ginf}, 
we consider the truncated spiral
\[
X'_k=\varphi((1/2,1)\times(0,2\pi J+2(k+1)\pi)). 
\]
Using the arguments above we see that $\textsf{E}(X'_k\times iB_r)$ 
has the required extension properties. 
Smoothness can be obtained by rounding off the corners by moving the boundary outside of $X'_k$,
which results in a slightly larger domain $X_k$. 
The construction of $X_k$ and the proof of Theorem \ref{main} are complete. \qed

\begin{remark}\rm
{\bf{(a)}} Since the family of curved rectangles depends continuously on the central angle,
we obtain a subset of $\textsf{E}(D_\infty)$ as parametrised in (\ref{lift}).
A careful examination of the proof of Lemma \ref{Ginf} shows that we can take $C=2J\pi+\pi/2$, $\rho=\eta/\sqrt{2}$
in (1).

{\bf{(b)}} A straightforward modification of the above proof, relying on the properties of Levi extension,
shows that we could have started from an \emph{arbitrary} tubular thickening
of the central spiral $\varphi(\{3/4\}\times\R_{>0})$. The resulting slight simplification will be used in the next section.
\end{remark}

\section{Universally covered circles}
Let $D\subset\C^n$ be a domain and $\gamma\subset\C^n$ an embedded loop 
(a curve homeomorphic to the circle)
with inclusion $\iota_\gamma:\gamma\hookrightarrow\C^n$.
A connected component $\gamma'$ of $\pi_D^{-1}(\gamma)$ is either {\bf (a)} a loop or {\bf (b)} an arc.
In case {\bf (a)}, $\pi_D|_{\gamma'}:\gamma'\rightarrow\gamma$ is a covering 
of $\gamma$ by $\gamma'$, whose topology is determined by the number 
$k$ of sheets over an arbitrary point of $\gamma$ (here $k$ is a positive integer).
Case {\bf (b)} splits again: 
{\bf (b$_1$)} 
$\pi_D|_{\gamma'}:\gamma'\rightarrow\gamma$ may be the universal covering of $\gamma$, i.e.~$k=\infty$.
{\bf (b$_2$)} 
Otherwise, $\pi_D|_{\gamma'}$ is not a covering, 
and the closure of $\gamma'$ in the abstract closure of $\textsf{E}(D)$ 
contains at least one point in the abstract boundary.

If the case {\bf (b$_1$)} is valid, we say also 
that $\gamma'$ \emph{universally covers $\gamma$} 
(and that $\gamma$ is \emph{universally covered by $\textsf{E}(D)$}).
Intuitively, this means that $\textsf{E}(D)$ becomes as multi-sheeted as possible above $\gamma$.
Note that the definition may be extended to higher-dimensional
submanifolds of $\C^n$ instead of $\gamma$.

The goal of this section is to strengthen the construction of $D_\infty$ 
in the proof of Theorem \ref{main}. 

\begin{proposition}
Assume that $\Pi=\R^2$ and $L$ are selected as in the construction of $X_2$ 
and that $R$ is defined as in (\ref{R}).
For every $r\in(0,R)$, there is a bounded domain $X\subset\R^2$ such that the envelope of holomorphy of
$D=X+iB_r$ universally covers  a circle $T\subset\R^2\backslash X$.
\end{proposition}

\begin{proof}
Since the idea is close to the proof of Theorem \ref{main}, we will only sketch the construction.
We choose $\Pi=\R^2$, $L$ and $R$ as in the construction of $X_2$.
Instead of working with a spiral $S_\epsilon$ accumulating from outside to $\{|x|=1\}$, 
we start from a spiral contained in a thin annulus 
\[
A_\epsilon=\{1<|x|<1+\epsilon\}\subset\R^2
\] 
and accumulating at both of its boundary circles
($\epsilon$ will later be chosen later).
Concretely, $S_\epsilon$ can be taken as the curve parametrised by
\[
\psi_\epsilon:\theta\mapsto\left(1+\epsilon\left(\frac{1}{2}+\frac{\arctan{\theta}}{\pi}\right)\right)(\cos\theta,\sin\theta),\ \theta\in\R.
\]

This spiral accumulates to $T=\{|x|=1\}$ for $\theta\rightarrow -\infty$ and to $\{|x|=1+\epsilon\}$ for $\theta\rightarrow\infty$.
The domain $X$ (which also depends on $\epsilon$) is now chosen as a tubular neighbourhood\footnote{
	i.e.~a domain $X\subset A_\epsilon$ containing $S_\epsilon$ and admitting a homeomorphism $S_\epsilon\times (-1,1)\rightarrow X$
	which restricts to $(x,0)\mapsto x$ 
	along $S_\epsilon\times\{0\}$.}
of $S_\epsilon$ in $A_\epsilon$. Its precise shape is not important, but
we can arrange that $X$ is smoothly bounded as a domain in $A_{\epsilon}$ (but not in $\R^2$).
\vspace*{-0.55cm}
\label{universal}
	\begin{figure}[h]
		\def\svgwidth{\linewidth}
		\fontsize{14pt}{1em} 
		\scalebox{.6}{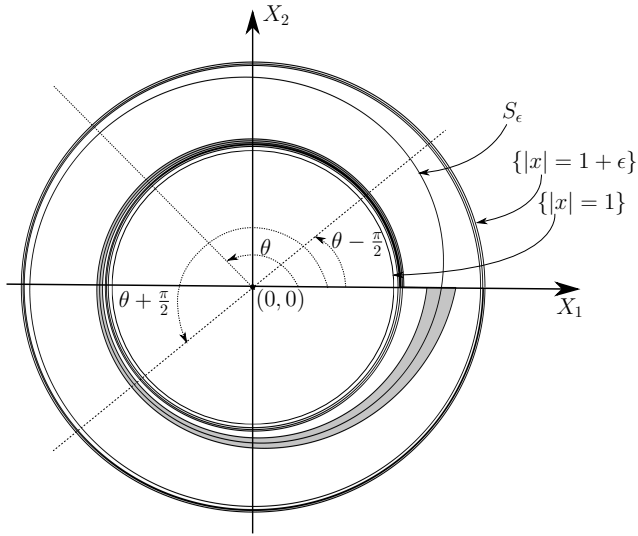}
			\vspace*{-0.4cm}
		\caption{Intersection of $D$ with $\R^2$}
	\end{figure}
Fix now $r\in(0,R)$. 
Similarly as before, we look at segments 
\[
I_\epsilon(\theta)=\psi_\epsilon(\theta-\pi/2,\theta+\pi/2)\subset S_\epsilon,\ \theta\in\R,
\] 
which correspond to a change of angle by $\pi$.
The key observation is that a sufficiently small choice of $\epsilon$ guarantees that for \emph{all} $\theta\in\R$,
holomorphic functions defined near  $I_\epsilon(\theta)+iB_r$ extend to a uniform domain
containing the point $p_\theta=(\cos\theta,\sin\theta)$ on the unit circle $T\subset\R^2$. 
More precisely, we may even arrange that this domain contains 
$L_\epsilon(\theta)=\big(I_\epsilon(\theta)\cup B^-_{2\epsilon}(\psi(\theta))\big)+iB_\epsilon$,
where $B^-_{2\epsilon}(\psi(\theta))$ is the connected component of 
$B^-_{2\epsilon}(\psi(\theta))\backslash I_\epsilon(\theta)$ containing $p_\theta$.

We conclude similarly as before. Looking at a univalent branch on $X+iB_r$ of the logarithm $f_L$ "around" $L$,
we see that every turn along $S_\epsilon$ brings us to a new sheet 
in the envelope of $D=X+iB_r$.
Formally, this can again be expressed by a mapping which associates to every $\theta\in\R$
the germ at $(\cos\theta,\sin\theta)$ of the holomorphic extension of the function restricted to a
neighbourhood of $I_\epsilon(\theta)+iB_r$.
Since the sets $L_\epsilon(\theta)$ move continuously with $\theta$,
$T$ is universally covered by $\textsf{E}(D)$. 
\end{proof}

It is crucial for the above construction that the thickness $\epsilon$ of the annulus decreases with $r$. 
To explain the reason, let us keep $X$ as in the proof but replace $D$ by $D'=X+iB_{r'}$ with $r'>0$ arbitrary. 
If $r'$ is small enough, the lines $L_x$ through the points $x\in T$ and parallel to $L$ are disjoint from 
\[
\{z\in\C^2:|x|=1+\epsilon,|y|\leq r'\}.
\]
Hence for $\theta$ sufficiently large, the hypersurface $M=I_\epsilon(\theta)+iB_{r'}$ does not intersect 
$L_{(\cos\theta,\sin\theta)}$. 
This excludes holomorphic extension from $M$ to a neighbourhood of $(\cos\theta,\sin\theta)$,
and the main argument of the proof breaks down.


\section{Schlichtness in dimension $2$}
In the above examples, multi-sheetedness was derived from Levi extension through pseudoconcave parts of the boundary of $D$.
Moreover, there were always pseudoconvex parts of the boundary over which multi-sheetedness was observed. 
The goal of the section is to provide some evidence showing that this interplay is \emph{necessary} for multi-sheetedness in dimension $2$.

Let $C$ be a smoothly embedded circle in $\R^2$. Picking a parametrisation of $C$ 
by a smooth, $2\pi$-periodic immersion $\gamma:\R\rightarrow\R^2$ the increment of
the angle may be viewed as a $2\pi$-periodic $1$-form $\alpha(\theta)\,d\theta$. 
We call $C$ \emph{strictly convex} if $\alpha(\theta)$ does not vanish
(i.e.~there are no inflection points). 
It is elementary to verify that this notion is well-defined and that the bounded domain $X_{C,\sf in}$
surrounded by a strictly convex circle $C$ is convex. 
By a \emph{convex domain with finitely many strictly convex holes} we mean a domain $X\subset\R^2$
obtained from a convex domain $X_{\sf cvx}$ by excision of finitely many disjoint closed discs of the form
$\overline{X}_{C_j,\sf in}$ where the $C_j$ are strictly convex circles.
We will keep these notations in the following arguments.\\

\noindent
{\bf Proof of Theorem \ref{schlicht2d}:} 
The essence of the proof is contained in the treatment of a special case,
which may be viewed as an appropriate deformation of the model domains (\ref{model}):
Assume that $X$ is bounded by two strictly convex circles $C_0,C_1$, 
which we number so that $X$ is strictly concave along $C_1$,
and set $D=X+iY$, with $Y$ as in the theorem.

For $K\subset\C^n$ compact, we recall that
\[
\widehat{K}=\{z\in\C^n:|p(z)|\leq\max_{K}|p| \mbox{ for each holomorphic polynomial } p\}
\] 
denotes its polynomial hull. We obtain the following description of $\textsf{E}(D)$ 
in terms of polynomial hulls. 

\begin{lemma}\label{polhull}
Let $K_x$ be the closure of the convex domain bounded by $C_1$ and $K=K_x+i\partial Y$.
Then $\textsf{E}(D)=D'\backslash\widehat{K}$, where $D'=(X\cup K_x)+iY$.
In particular, $\widehat{K}$ does not disconnect $D'$.
\end{lemma}

The lemma may be viewed as a variant of a classical result by Stout \cite{Sto}
on extension from parts of strictly pseudoconvex boundaries. 
For the reader's convenience, we sketch an independent proof
along the lines of \cite{MP2}, see also \cite{LP} for a related application in dimension $2$.\\

\noindent
{\bf Proof of Lemma \ref{polhull}:} 
There is a nonnegative smooth function $\varphi\in\CalC^\infty(\C^2)$
that vanishes precisely on $\widehat{K}$ and is strictly plurisubharmonic on the complement.
Since $\widehat{K}\subset K_x+i\overline{Y}$, we have $\varphi>0$ on $D$.
After an appropriate deformation of $\varphi$, we can in addition assume 
that $\varphi|_{\C^2\backslash\widehat{K}}$ is a good Morse function\footnote{
i.e.~a function whose critical points are isolated nondegenerate quadratic singularities lying on different level sets.}.
Next we deform $C_1\times i\R^2$ such that $C_1+i(\R^2\backslash Y)$ is fixed
and $C_1\times Y$ is moved into $D$.
Since $C_1\times Y$ is strictly pseudoconvex hypersurface, 
a $\CalC^2$-small deformation of this kind will yield 
a strictly pseudoconvex hypersurface $M\subset D$.
Moreover, $M\cup K$ bounds a domain $D_{\sf in}\subset D'$,
which is slightly larger than $\mbox{int}(K_x)+iY$. 
Then $M$ disconnects $D_{\sf in}$ into $D''$ and an outer domain $D_{\sf out}$.
Finally we may arrange in addition that $\varphi|_M$ is a good Morse function.

We use the methods from \cite{MP2} to extend (restrictions to $D_{\sf out}$  of) 
holomorphic functions on $D$ to the open sets 
\[
Q_{h}=D_{\sf out}\cup\{z\in D_{\sf in}:\varphi(z)>h\},\ h>0.
\]
We need the following three elementary properties: 
\begin{itemize}
\item[(i)] $Q_h=D_{\sf out}$ for $h\gg 0$,
\item[(ii)] $Q_{h_1}\subset Q_{h_2}$ if $h_1>h_2$,
\item[(iii)] $\bigcup_{h>0}Q_h=D'\backslash\widehat{K}$.
\end{itemize}

The idea is to construct extension by continuously letting decrease $h\downarrow 0$.
Assume that the extension to $Q_{h_0}$ is valid for some $h_0>0$.
If $h_0$ is neither a critical value of $\varphi$ nor of $\varphi|_M$,
it is not hard to construct extension to $Q_{h_1}$ for some $h_1<h_0$ 
by gluing finitely many local Levi extensions through the strictly pseudoconvex
hypersurface $\{\varphi=h_0\}$. 
The general case requires a careful study of the local geometry at Morse singularities.
For the details, we refer to \cite{MP2}. 
Note however that the constructions in \cite{MP2}, where extension from  \emph{non}pseudoconvex boundaries is considered,
have to take transitional multi-sheetedness into account.
As explained in \cite{LP}, these monodromy problems are easier to handle in our setting, thanks to the strict pseudoconvexity of $M$.

After constructing extension to $D'\backslash\widehat{K}$, the proof will be ready 
as soon as we have shown that $D'\backslash\widehat{K}$ is pseudoconvex.
This follows from the convexity of $D'$ and the known fact that the polynomial hull $\widehat{K}$ is pseudoconcave 
at points of the essential hull $\widehat{K}\backslash K$, see for example \cite{Slo}. 
\qed

\begin{remark}\label{n>=3}
\rm
Most of the above proof generalises without change to $n\geq 3$, 
giving univalent extension from $D$ to $D'\backslash\widehat{K}$.
At least for the general case of extension from parts of strictly pseudoconvex boundaries, 
it is however known that this need not give the entire envelope and that the envelope
may even be multi-sheeted. \qed
\end{remark}

Let us now consider the general case, where we denote by $C_1,\ldots,C_m$ 
the components of $\partial X$ along which $X$ is strictly concave.
For each $\mu=1,\ldots,m$, we select a similar curve $C_{\mu,1}$ obtained by smoothly deforming $C_\mu$ slightly into $X$.
%
%
%
By Lemma  \ref{polhull} applied to the domain $X_\mu$ squeezed between $C_\mu$ and $C_{\mu,1}$,
we have 
\[
\textsf{E}(X_\mu+iY)=X'_\mu\backslash\widehat{K_\mu}.
\]
Here $K_\mu=K_{x,\mu}+i\partial Y$ with $K_{x,\mu}$ being the closure of the convex domain bounded by $C_\mu$,
and $X'_\mu=(X_\mu\cup K_{x,\mu})+iY$.
By construction,
\[
\widehat{D}=D\cup\bigcup_{\mu=1}^m \textsf{E}(X_\mu+iY)
\]
embeds into the envelope of $D$.
To conclude the proof, it just remains to show that $\widehat{D}$ is pseudoconvex 
(this is the only place where we use the properties of the outer boundary component of $\partial X$)
and must therefore coincide with the envelope of holomorphy of $D$. 
Pseudoconvexity follows from the convexity of 
\[
D'=(X\cup\bigcup_{\mu=1}^m K_{x,\mu})+iY
\]
and the local pseudoconvexity along the essential hulls $\widehat{K_\mu}\backslash K_\mu$,
used already in the proof of Lemma \ref{polhull}.
The proof of Theorem \ref{schlicht2d} is complete.
\qed 

While the assumption on the strict convexity of the holes can probably be weakened,
we cannot completely omit an assumption on their shape.

\begin{example}\label{convexity}\rm
The following example is a variant of our first option of $D_2$ in Section \ref{sec2}.
Choose $\Pi=\R_x^2$ and a complex line $L=\{az_1+bz_2=0\}$ like there.
We construct $X'$ by adding to the annulus $A=\{(x_1,x_2)\in\R^2:2<|x|<3\}$
the bridge $\{|x_1|<1/2\}\cap B_3$ from which we cut away a sufficiently thin bent channel
as indicated in figure \ref{ncvx}
(the width of the channel corresponding to $\eta$ in the construction of $D_2$).
 
\begin{figure}[htbp]
	\centering
	\def\svgwidth{\linewidth}
		\fontsize{14pt}{1em}
	\scalebox{.6}{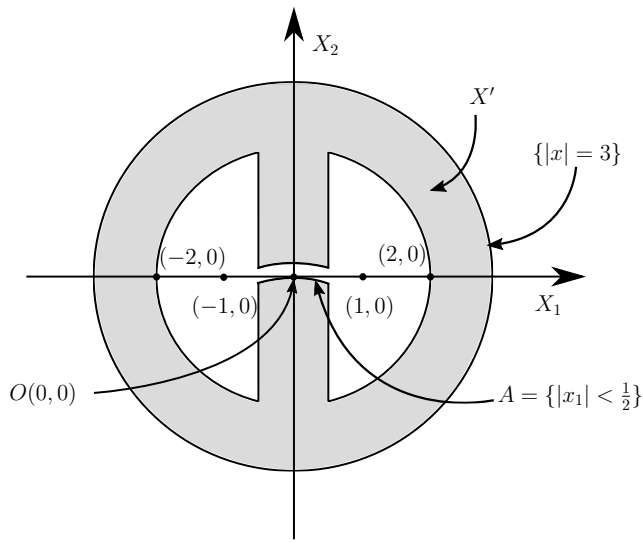}
	\label{ncvx}
		\vspace*{-0.4cm}
	\caption{The domain $X'$ in $\R^2$}
\end{figure}

We choose $r$ such that the closure of $D'=X'+iB_r$ does not intersect $L_{-1}$ and $L_1$,
the parallel translates of $L$ passing through the points $(\pm 1,0)$ respectively.
Now the choice of a branch of
\begin{equation}\label{root}
\sqrt{(a(z_1+1)+bz_2)(a(z_1-1)+bz_2)}
\end{equation}
defines a univalent holomorphic function $f$ on $D'$. 
If the channel is sufficiently thin, $f$ extends from the above part of the interrupted bridge to a neighbourhood of the origin.
Since a single turn around one of the branch points of (\ref{root}) changes the value we have exhibited two sheets over $0$.
Rounding off the eight corners of $X'$ we obtain an example with a nonconvex hole and a nonschlicht envelope of holomorphy.
\qed
\end{example}

It is instructive to relate Lemma \ref{polhull} to the results in \cite{JP2}.

\begin{example}\label{JPexp}\rm
{\bf a)} For $D_{2,r_1,r_2,r_3}$ like in (\ref{model}), 
it is enough to find the polynomial hull of the solid torus 
$K=\overline{B}_{r_1}+i\partial B_{r_3}$.
By convexity $\widehat{K}$ is contained in the closure of the polydisc tube 
$P=\overline{B}_{r_1}+i\overline{B}_{r_3}$.
The function $g=e^{z_1^2+z_2^2}$ has modulus $|g|=e^{|x|^2-|y|^2}$. 
Thus $\max_{K}|g|=r_1^2-r_3^2$ is obtained at the boundary $\partial B_{r_1}+i\partial B_{r_3}$ of $K$,
and we get
\[
\widehat{K}\subset L=\{|x|^2-|y|^2\leq r_1^2-r_3^2\}\cap P.
\]
To show that this is an equality we only have to observe 
that $L$ is foliated by the complex curves 
$\{z_1^2+z_2^2=c\}$, $\mbox{Re}(c)\leq r_1^2-r_3^2$
(pairs of discs if $\mbox{Re}(c)=\mbox{Im}(c)$ and annuli else),
with boundary in $K$. Equality then follows from the
maximum modulus principle.
Hence we have rederived (\ref{ED}) from Lemma \ref{polhull}.

For $n\geq 3$, our arguments concerning polynomial hulls are still valid.
Thus Remark \ref{n>=3} yields holomorphic extension to the set in (\ref{ED}),
and then we conclude by directly verifying that this set is pseudoconvex
as in \cite{JP2}. 
Note however that local pseudoconvexity of the complement of $\widehat{K}$
at points of the essential hull $\widehat{K}\backslash K$ is not true in general for $n\geq 3$.
The hull is still locally $1$-pseudoconcave, but this implies less pseudoconvexity for the complement.

{\bf b)} Looking at the "breaking point" $r_1=r_3$ leads us to considering 
$$\Delta=\{z=x+iy\in\C^2:|y|^2<|x|^2<1\}.$$
With respect to the standard real euclidean structure of $\C^2=\R^4$, 
it is congruent to the Hartogs triangle $\{z=(z_1,z_2)\in\C^2:|z_2|^2<|z_1|^2<1\}$.
However, its function theoretic properties are very different.
For example, since $\overline{\Delta}$ is polynomially convex,
$\Delta$ has no Nebenh\"ulle, in contrast to the Hartogs triangle.
\qed
\end{example}

It is rare that polynomial hulls can be described as explicitly as in the Example \ref{JPexp}. 
Therefore there is not much hope for more explicit information than in Theorem \ref{schlicht2d}. 
It is however possible to deduce qualitative properties like 

\begin{corollary}
Let $D$, $K_x$ and $Y$ be like in Lemma \ref{polhull}. Assume moreover that $0\in Y$,
and set $Y_R=R\cdot Y$, $K_R=K_x+\partial Y_R$, $D_R=D+iY_R$.
Then the following properties hold:
\begin{itemize}
\item[\bf (a)]
For every $R_0\geq 0$ there is $R_1>0$ such that 
$\widehat{K}_R\cap\{|y|\leq R_0\}=\emptyset$
if $R\geq R_1$.
\item[\bf (b)] $\textsf{E}(D+i\R^2)=\bigcup_{R>0}\textsf{E}(D_R)$.
\end{itemize}
\end{corollary}

{\bf Proof:} 
General arguments imply actually that {\bf (b)} always holds 
for a family of subdomains exhausting $D$ in a suitable sense 
if its is know that all envelopes of the subdomains are schlicht.
Alternatively, it is easy to see that {\bf (a)} and {\bf (b)} are equivalent, 
again because of schlichtness of the envelopes.

To verify {\bf (a)}, consider again the Gaussian functions
\[
f_{\zeta}(z)=\exp\big((z_1-\zeta_1)^2+(z_2-\zeta_2)^2\big).
\]
For $R_0$ fixed, a direct verification very similar to what is done in Example \ref{JPexp}
shows that
$\max_{K_R}|f_{\zeta}|<c<1$ holds for all $\zeta =(\zeta_1,\zeta_2)\in K_x+i\overline{B}_{R_0}$,
provided $R$ is sufficiently large.
Since $f_{\zeta}(\zeta)=1$, we obtain {\bf (a)}. \qed

\end{document}